\documentclass{amsart}

\usepackage[colorlinks]{hyperref}
\usepackage[capitalize,nameinlink,noabbrev]{cleveref}

\newtheorem{theorem}{Theorem}[section]
\newtheorem{lemma}[theorem]{Lemma}
\newtheorem{corollary}[theorem]{Corollary}
\newtheorem{proposition}[theorem]{Proposition}

\theoremstyle{definition}
\newtheorem{definition}[theorem]{Definition}
\newtheorem{example}[theorem]{Example}

\newtheorem{remark}[theorem]{Remark}

\newtheorem{theoremx}{\bf Theorem}

\numberwithin{equation}{section}
\usepackage{amscd, amssymb, enumerate, float, mathtools}
\DeclarePairedDelimiter\ceil{\lceil}{\rceil}

\newcommand{\N}{\mathbb{N}}
\newcommand{\Z}{\mathbb{Z}}
\newcommand{\Q}{\mathbb{Q}}
\newcommand{\p}{\mathfrak{p}}

\newcommand{\ux}{\underline{x}}
\newcommand{\ua}{\underline{\alpha}}
\newcommand{\ub}{\underline{\beta}}

\begin{document}

\title{Limit Behavior of The Rational Powers of Monomial Ideals}

\author{James Lewis}
\address{James Lewis \\ Department of Mathematical Sciences  \\ New Mexico State University  \\PO Box 30001\\Las Cruces, NM 88003-8001}
\curraddr{}
\email{jglewis@nmsu.edu}
\thanks{}


\date{}

\dedicatory{}

\begin{abstract}
We investigate the rational powers of ideals. We find that in the case of monomial ideals, the canonical indexing leads to a characterization of the rational powers yielding that symbolic powers of squarefree monomial ideals are indeed rational powers themselves. Using the connection with symbolic powers techniques, we use splittings to show the convergence of depths and normalized Castelnuovo–Mumford regularities. We show the convergence of Stanley depths for rational powers, and as a consequence of this we show the before-now unknown convergence of Stanley depths of integral closure powers. In addition, we show the finiteness of asymptotic associated primes, and we find that the normalized lengths of local cohomology modules converge for rational powers, and hence for symbolic powers of squarefree monomial ideals.
\end{abstract}

\maketitle

\section{Introduction}

In a commutative ring $A$ with ideal $I$ and $\frac{a}{b}\!\in\!\Q_+$ we define \textit{the $\frac{a}{b}$-rational power of $I$} to be
\[
I^{\frac{a}{b}}\!=\!\{x\!\in\! A |\ x^b\!\in\!\overline{I^a}\}.
\]
Despite having been formally defined in the literature nearly two decades ago, these simplistically defined powers are sparsely discussed with \cite{ciu20,deStefani16,HonUlr14,HSInt,rush07,taylor20} being some of the only articles which discuss rational powers. For instance, in the Noetherian case we know that the family $\{I^{\beta}\}_{\beta\in\Q_+}$ of rational powers can be linearly indexed by $\N$, i.e. the rational powers form a filtration. However, nothing has been said about the homological invariants of this filtration.

In this article we show the canonical indexing of the rational powers -- i.e. $\{I^{\frac{n}{e}}\}_{n\in\N}$ where $e$ is the least common multiple of the Rees valuations of $I$ (whose set we denote $\mathcal{RV}(I)$) evaluated at $I$ -- leads to a characterization of the rational powers which allows for use of the techniques often found in arguments involving the symbolic powers of squarefree monomial ideals (e.g. in \cite{Montano18,Fak17}). The characterization -- having a uniform bound in the Rees valuation inequalities -- allows us, in the case of squarefree monomial ideals, to find that the symbolic powers are in fact rational powers of some ideal. In general, we can use hyperplanes to characterize rational powers of monomial ideals via the following theorem (see \cref{bigThm}) where a hyperplane is given by, for $\underline{X}\in\Q^d$, $h(\underline{X}) = \underline{a}\cdot\underline{X}$ for coefficients $\underline{a}\in\Q^d$:

\begin{theoremx}
\label{characterizeRP}
\textit{Let hyperplanes $h_1,\dots,h_r$ with coefficients in $\Q_{\geqslant 0}^d$ define a family of monomial ideals $\{I_{\sigma}\}_{\sigma\in\Q_+}$ in $R\!=\!\mathbb{K}[x_1,\dots,x_d]$ so that $\ux^{\ua}\!\in\! I_{\sigma}$ if and only if $h_i(\ua)\!\geqslant\!\sigma$ for all $1\!\leqslant\! i\!\leqslant\! r$. Then there exists a monomial ideal $J$ and $g\!\in\!\N$ so that for any $\sigma\!\in\!\Q_+$ we have $I_{\sigma}\!=\!J^{\frac{\sigma}{g}}$.}
\end{theoremx}

The connection with symbolic powers leads to a generalization of the splitting maps method found in \cite{MNB20,Montano18} which guarantees the convergence of depths and normalized Castelnuovo–Mumford regularities (see \cref{depthReg} and \cref{verLimits}): 

\begin{theoremx}
\label{regdepthRP}
\textit{If $I$ is any monomial ideal and $e\!=\!\textnormal{lcm}\{v(I)\,|\,v\!\in\!\mathcal{RV}(I)\}$, then 
\begin{enumerate}
    \item[$(1)$] $\lim\limits_{n \to \infty}\frac{\textnormal{reg}(I^{\frac{n}{e}})}{n}$ exists and is equal to $\frac{1}{e}\lim\limits_{n \to \infty}\frac{\textnormal{reg}(\overline{I^n})}{n}$.
    \item[$(2)$] $\lim\limits_{n \to \infty}\textnormal{depth}(R/I^{\frac{n}{e}})$ exists and is equal to $d-\ell(I)$, where $\ell(I)$ is the analytic spread of $I$.
\end{enumerate}}
\end{theoremx}

This computation also yields that the symbolic analytic spread (as discussed in \cite{DaoMon19}) can be computed via the symbolic polyhedron (as discussed in \cite{CoopSyP}) in the case of squarefree monomial ideals (see \cref{symSpread}).

Furthermore, we show the convergence of Stanley depths (a measure of the (multi-)graded structure of the filtration) for rational powers using methods from \cite{Fak20}. As a consequence we answer the open question (from personal communications \cite{FakEmail} and discussed in \cite{Fak19}) that the Stanley depths must also converge for integral closure powers. (see \cref{sdepth}):

\begin{theoremx}
\label{sdepthRP}
\textit{If $I$ is a monomial ideal and $e\!=\!\textnormal{lcm}\{v(I)\,|\,v\!\in\!\mathcal{RV}(I)\}$. Then the limits $\lim\limits_{n \to \infty}\textnormal{sdepth}(R/I^{\frac{n}{e}})$ and $\lim\limits_{n \to \infty}\textnormal{sdepth}(I^{\frac{n}{e}})$ exist. In particular, these limits must exist for $\{\overline{I^n}\}_{n\in\N}$. Furthermore:
\begin{equation*}
\begin{split}
    \lim\limits_{n \to \infty}\textnormal{sdepth}(R/I^{\frac{n}{e}})\!=\!\min_n\textnormal{sdepth}(R/\overline{I^n})\textnormal{, and} \\
    \lim\limits_{n \to \infty}\textnormal{sdepth}(I^{\frac{n}{e}})\!=\!\min_n\textnormal{sdepth}(\overline{I^n}).
\end{split}
\end{equation*}}
\end{theoremx}

Herzog conjectured the existence of the Stanley depth limits for regular powers in \cite[Conjecture 59]{Herzog13sd}, hence we confirm the conjecture in the case of normal ideals. 

One of the common factors in proving the aforementioned limits is that the associated Rees algebra of the rational powers is Noetherian. Combining this with the fact that rational powers are integrally closed, we conclude that $\cup_{n\in\N}\textnormal{Ass}(R/I^{\frac{n}{e}})$ is finite (see \cref{assInt}).

Furthermore we use the methods of \cite{DaoMon17} to find that the lengths of local cohomology modules involving the rational powers of a monomial ideal also converge (see \cref{lcLim}):

\begin{theoremx}
\label{lengthRP}
\textit{Let $I$ be a monomial ideal and $e\!=\!\textnormal{lcm}\{v(I)\,|\,v\!\in\!\mathcal{RV}(I)\}$. Assume that $\lambda(\textnormal{H}_{\mathfrak{m}}^i(R/I^{\frac{n}{e}}))\!<\!\infty$ for $n\!\gg\!0$. Then the limit
\[
\lim\limits_{n \to \infty}\frac{\lambda(\textnormal{H}_{\mathfrak{m}}^i(R/I^{\frac{n}{e}}))}{n^d}
\]
exists and is rational. In particular, this limit must exist for $\{I^{(n)}\}_{n\in\N}$ when $I$ is squarefree. Furthermore $\lim\limits_{n \to \infty}\frac{\lambda(\textnormal{H}_{\mathfrak{m}}^i(R/I^{\frac{n}{e}}))}{n^d}\!=\!\frac{1}{e^d}\lim\limits_{n \to \infty}\frac{\lambda(\textnormal{H}_{\mathfrak{m}}^i(R/\overline{I^n}))}{n^d}$.}
\end{theoremx}

\section{Preliminaries}

We now outline preliminaries of integral closure and valuation theory that we use throughout the article. For full details and proofs we refer the reader to \cite[Chapters 1,6, and 10]{HSInt} and \cite{HubSwa08}.

For an ideal $I$ in a commutative Noetherian ring $A$, we define the \textit{integral closure} of $I$ to be the ideal
\[
\overline{I}\!=\!\{x\!\in\! A\,|\,x^n+a_1x^{n-1}+\dots+a_{n-1}x+a_n\!=\!0 \textnormal{ for some }n\!\in\!\N \textnormal{, } a_i\!\in\! I^i \textnormal{ for }1\!\leqslant\! i\!\leqslant\! n\}.
\]
Using the theory of valuations, integral closure becomes easier to compute and work with. To begin with valuations:

Let $K$ be a field, $K^*\!=\!K-\{0\}$ the multiplicative group. A \textit{discrete (rank one) valuation} on $K$ is a group homomorphism $v:K^*\to \Z$ with the added property that $v(x+y)\!\geqslant\!\textnormal{min}\{v(x),v(y)\}$ for any $x,y\!\in\! K^*$.
We can associate to each valuation $v$ a \textit{valuation ring} $R_v\!=\!\{x\!\in\! K\,|\,v(x)\!\geqslant\!0\}\cup\{0\}$ which is a local domain. We call a valuation, $v$, \textit{normalized} if $v(K^*)\subseteq\Z$ has greatest common divisor one. 

For an ideal $I$ in a domain $A$ and a discrete (rank one) valuation, $v$, on the fraction field of $A$, $Q(A)$, we define
\[
v(I)\!=\!\textnormal{min}\{v(x)\,|\,x\!\in\! I\}.
\]
We may consider only the set of generators of $I$ to compute $v(I)$. One may show further that $v(I^n)\!=\!nv(I)$.

There is a powerful connection between integral closures and valuations which says that, if $A$ is a domain, for an ideal $I$ and for any $n\!\in\!\N$ we have the \textit{valuative criterion} of integral closure
\[
x\!\in\!\overline{I^n}\textnormal{ if and only if }v(x)\!\geqslant\! nv(I)
\]
for all discrete valuations of rank one $v$ with $R_v$ between $A$ and $Q(A)$ and such that the maximal ideal of $R_v$ contracts to a maximal ideal of $A$. This connection also shows that $v(\overline{I})\!=\!v(I)$.

From a construction of Rees, for a given ideal $I$ there exist a finite set of unique (up to equivalence of valuations) normalized discrete rank one valuations for which we need to check the valuative inequality for integral closure powers. We call these valuations the \textit{Rees valuations} of $I$ and denote the set of them by $\mathcal{RV}(I)$. That is,
\[
x\!\in\!\overline{I^n}\textnormal{ if and only if }v(x)\!\geqslant\! nv(I)\textnormal{ for all }v\!\in\!\mathcal{RV}(I)
\]
where, critically, $\mathcal{RV}(I)$ is finite.

\begin{definition}
In a commutative ring $A$ with ideal $I$ and $\frac{a}{b}\!\in\!\Q_+$ we define \textit{the $\frac{a}{b}$ rational power of $I$} to be
\[
I^{\frac{a}{b}}\!=\!\{x\!\in\! A |\ x^b\!\in\!\overline{I^a}\}.
\]
\end{definition}
Surprisingly this set turns out to be an ideal. The fact of it being an ideal follows from the properties of valuations and the valuative criterion of integral closure above. The following theorem outlines some basic facts about rational powers in a Noetherian ring. 
\begin{theorem}
\label{basics}
\cite[Section 10.5]{HSInt} For a commutative Noetherian ring $A$, ideal $I$, and $\alpha,\beta\!\in\!\Q_+$, set $e\!=\!\textnormal{lcm}\{v(I)\,|\,v\!\in\!\mathcal{RV}(I)\}$ where lcm denotes least common multiple, then:
\begin{enumerate}
    \item[$(1)$] If $a,b,c,d\!\in\!\N$, then $I^{\frac{a}{b}}$ is a well-defined ideal, that is, if $\frac{a}{b} \!=\! \frac{c}{d}$ then $I^{\frac{a}{b}} \!=\! I^{\frac{c}{d}}$
    \item[$(2)$] $I^{\frac{n}{1}}\!=\!\overline{I^n}$
    \item[$(3)$] if $\alpha \!\leqslant\! \beta$ then $I^{\beta} \subseteq I^{\alpha}$
    \item[$(4)$] $I^{\alpha}$ is integrally closed
    \item[$(5)$] $x\!\in\! I^{\alpha}$ if and only if $v(x)\!\geqslant\! \alpha v(I)$ for every Rees valuation $v$ of $I$
    \item[$(6)$] $I^{\alpha}I^{\beta}\subseteq I^{\alpha+\beta}$
    \item[$(7)$] for all $\alpha\!\in\!\Q_+$, $I^{\alpha} \!=\! I^{\frac{n}{e}}$ where $n\!=\!\ceil*{\frac{ea}{b}}$
\end{enumerate}
\end{theorem}

Of particular interest is the second property that will allow us to move back and forth between rational and integral closure powers, so that some features of rational powers will be able to pull back to the integral closure powers. We also note that the last property allows us to describe the rational powers of an ideal as an $\N$-indexed filtration of ideals. By \textit{filtration} we mean a decreasing chain of ideals $\{I_n\}_{n\in\N}$ such that $I_0\!=\!A$ and $I_n\cdot I_k\subseteq I_{n+k}$ for any $n,k\!\in\!\N$.

From now on let $R\!=\!\mathbb{K}[x_1,\dots,x_d]$ for some field $\mathbb{K}$ and $d\!\in\!\N$ and $I$ be a monomial ideal of $R$. We will denote a general monomial of $R$ by $\ux^{\ua}\!=\!x_1^{\alpha_1}\dots x_d^{\alpha_d}$ where $\ua\!=\!(\alpha_a,\dotsm,\alpha_d)\!\in\!\N^d$. Then we say the \textit{exponent set} of $I$ is
\[
G(I)\!=\!\{\ua\!\in\!\N^d\,|\,\ux^{\ua}\!\in\! I\}
\]
and we call the \textit{Newton Polyhedron} the convex hull of the exponent set of $I$, that is
\[
\textnormal{NP}(I)\!=\!\textnormal{Conv}(G(I)).
\]
The Newton Polyhedron connects convex geometry to the integral closure via \cite[Proposition 1.4.6]{HSInt} so that for any monomial ideal $I$ we have
\[
G(\overline{I})\!=\!\textnormal{NP}(I)\cap\N^d.
\]
That is, we can read-off the integral closure of an ideal via the lattice points of the Newton Polyhedron.

The connection between convex geometry and integral closure can also be extended to the valuation theory of integral closures.
To set this up: a valuation, $v$, on the field of fractions of $R$ is called \textit{monomial} if for any polynomial $f$ we have that $v(f)\!=\!\textnormal{min}\{v(\ux^{\ua})\,|\,\ux^{\ua}$ is a monomial supporting $f\}$. Then we have that the Rees valuations of a monomial ideal are monomial (see \cite[Theorem 10.3.4]{HSInt}). 

We say a \textit{hyperplane} in $\Q^d$ is a function $h$ on $\Q^d$ defined by $h(\underline{X}) = \underline{X}\cdot\underline{\alpha}$ for coefficients $\underline{\alpha}\in\Q^d$

\begin{remark}
\label{hyperP}
Now, let $v$ be a monomial valuation on $R\!=\!\mathbb{K}[x_1,\dots,x_d]$. The valuation is determined by how it behaves on monomials and so, since valuations are multiplicative group homomorphisms,
\[
v(x_1^{a_1}\dots x_d^{a_d})\!=\!a_1v(x_1)+\dots+a_dv(x_d)
\]
shows
the valuation is given by the hyperplane $v(x_1)X_1+\dots+v(x_d)X_d$ in $\Q_{\geqslant 0}^d$.
\end{remark}

Hence it may be no surprise that the Rees valuations of a monomial ideal are given by the bounding faces of the Newton polyhedron (see \cite[Theorem 10.3.5]{HSInt}). That is, we can read-off the Rees valuations of a monomial ideal by the hyperplanes which make up the Newton polyhedron, and furthermore the valuations given by the hyperplanes are normalized from \cite[Corollary 3.3]{HubSwa08}. 

For $n\!\in\!\N$ we define the \textit{nth-symbolic power} of an ideal $I$ to be
\[
I^{(n)}\!=\!\bigcap_{\p\in\textnormal{Min}(I)}I^nR_{\p}\cap R
\]
where $\textnormal{Min}(I)$ denotes the minimal primes of $I$.

For any graded $R$-module $M$ set
\[
a_i(M)\!=\!\max\{j|\,\textnormal{H}_{\mathfrak{m}}^i(M)_j\neq 0\}
\]
for $0\!\leqslant\! i\!\leqslant\! \textnormal{dim }M$, we call $a_i$ the \textit{$i$th $a$-invariant}. Then the \textit{Castelnuovo–Mumford regularity} of $M$ is
\[
\textnormal{reg}(M)\!=\!\textnormal{max}\{a_i(M)+i\,|\,0\!\leqslant\! i\!\leqslant\! \textnormal{dim }M\}.
\]

Another measure of the graded structure of a module is the Stanley Depth. For a survey of Stanley depth, we refer the reader to \cite{PFTY09}. Let $M$ be a finitely generated $\Z^d$-graded $R$-module, $Z$ a subset of the variables (i.e. $Z\subseteq\{x_1,\dots,x_d\}$), and $u\!\in\! M$ homogeneous. If $u\cdot \mathbb{K}[Z]$ is a free $\mathbb{K}[Z]$-module, then we call $u\cdot \mathbb{K}[Z]$ a \textit{Stanley space} of dimension $|Z|$. A $k$-vector spaces decomposition, $\mathcal{D}$, of $M$ into a direct sum of Stanley spaces is called a \textit{Stanley decomposition} of $M$. We call the \textit{Stanley depth} of the decomposition the minimum dimension of a Stanley space appearing in $\mathcal{D}$ and is notated $\textnormal{sdepth}(\mathcal{D})$. That is, if we can write $\mathcal{D}$ as the direct sum of $k$-vector spaces:
\[
M\!=\!u_1\cdot \mathbb{K}[Z_1]\oplus\dots\oplus u_r\cdot \mathbb{K}[Z_r]
\]
where each $u_i\cdot \mathbb{K}[Z_i]$ is a free $\mathbb{K}[Z_i]$-module then we say
\[
\textnormal{sdepth}(\mathcal{D})\!=\!\textnormal{min}\{|Z_i|\,|\,1\!\leqslant\! i\!\leqslant\! r\}.
\]
The maximum of the Stanley depths over all Stanley decompositions of $M$ is called the \textit{Stanley depth} of $M$. That is,
\[
\textnormal{sdepth}(M):=\!\textnormal{max}\{\textnormal{sdepth}(\mathcal{D})\,|\,\mathcal{D}\textnormal{ is a Stanley decomposition of } M\}.
\]

\section{Characterizing Rational Powers}

In this section we extend the valuative criterion of integral closures to help characterize rational powers. Notice that we have a version of the valuative criterion in \cref{basics} using rational numbers. Recall that $R\!=\!\mathbb{K}[x_1,\dots,x_d]$ for some field $\mathbb{K}$ and $d\!\in\!\N$, and we have $I$ a monomial ideal of $R$.

\begin{remark}
\label{MVPCrit}
Using the $e\!=\!\textnormal{lcm}\{v(I)\,|\,v\!\in\!\mathcal{RV}(I)\}$ indexing of the rational powers, we have that $u\!\in\! I^{\frac{n}{e}}$ if and only if $v(u)\!\geqslant\! \frac{n}{e}v(I)$ for each $v\!\in\!\mathcal{RV}(I)$. As before, we can rearrange this to 
\[
u\!\in\! I^{\frac{n}{e}}\textnormal{ if and only if }\frac{e}{v(I)}v(u)\!\geqslant\! n\textnormal{ for each }v\!\in\!\mathcal{RV}(I).
\]
Since $\frac{e}{v(I)}\!\in\!\N$ by choice of $e$, if we need to check membership to a rational power, we need only check that a finite number of functions with values in $\N$ are uniformly bounded below by $n$. Critically, this allows for use of the techniques often found in arguments involving the symbolic powers of squarefree monomial ideals (e.g. in \cite{Montano18,Fak17}), see \cref{regstable}.
\end{remark}

We now turn our attention to monomial ideals where the Rees valuations are given by hyperplane equations. Because of this added convex geometry, almost any filtration of ideals given by hyperplane equations can be described as rational powers. We note here that convex geometry yielding properties about a filtration has seen some recent advancements, for example in \cite{CoopSyP}. We now prove \cref{characterizeRP}.

\begin{theorem}
\label{bigThm}
Let non-redundant hyperplanes $h_1,\dots,h_r$ with coefficients in $\Q_{\geqslant 0}^d$ define a family of monomial ideals $\{I_{\sigma}\}_{\sigma\in\Q_+}$ in $R\!=\!\mathbb{K}[x_1,\dots,x_d]$ so that $\ux^{\ua}\!\in\! I_{\sigma}$ if and only if $h_i(\ua)\!\geqslant\!\sigma$ for all $1\!\leqslant\! i\!\leqslant\! r$
Then 
\begin{enumerate}
    \item[$(1)$] There exists a monomial ideal $J$ and $g\!\in\!\N$ so that for $\sigma\!\in\!\Q_+$ we have $I_{\sigma}\!=\!J^{\frac{\sigma}{g}}$.
    \item[$(2)$] The family $\{I_{\sigma}\}_{\sigma\in\Q_+}$ can be indexed by $\N$ so that it can be written as $\{I_{\frac{n}{f}}\}_{n\in\N}$ for some $f\!\in\!\N$.
    \item[$(3)$] The indexing of the rational powers of $J$ with $e\!=\!\textnormal{lcm}\{w(J)\,|\,w\!\in\!\mathcal{RV}(J)\}$, $\{J^{\frac{n}{e}}\}_{n\in\N}$ appears as a subsequence of $\{I_{\frac{n}{f}}\}_{n\in\N}$ via $J^{\frac{n}{e}}\!=\!I_{\frac{fg}{e}\cdot\frac{n}{f}}$ where $\frac{fg}{e}\!\in\!\N$.
\end{enumerate}
\end{theorem}
\begin{proof}
We begin with the proof of (1). 
We will show that the polyhedron formed by the hyperplanes defining the family $\{I_{\sigma}\}_{\sigma\in\Q_+}$ can be scaled to be the Newton Polyhedron of some monomial ideal. Without loss of generality we can write each hyperplane $h_i\!=\!1$ in the reduced integral form as $a_1^iX_1+\dots+a_d^iX_d\!=\!f_i$ where $a_j^i,f_i\!\in\!\N$ for $1\!\leqslant\! i \!\leqslant\! r$ and $1\!\leqslant\! j\!\leqslant\! d$ and with $\textnormal{gcd}(a_1^i,\dots,a_d^i,f_i)\!=\!1$ for $1\!\leqslant\! i \!\leqslant\! r$. Let $f\!=\!\textnormal{lcm}(f_1,\dots,f_d)$.

Consider the convex hull, $C_1$, of the hyperplanes $h_1\!=\!f_1,\dots,h_r\!=\!f_r$. Since the coefficients of the hyperplanes are positive, $C_1$ is the set of all the points $\ub\!\in\!\Q_+^d$ with $h_1(\ub)\!\geqslant\! f_1,\dots,\textnormal{ and }h_r(\ub)\!\geqslant\! f_r$. Let $C_t$ be the convex set given by all the points $\ub\!\in\!\Q_+^d$ with $h_1(\ub)\!\geqslant\! f_1t,\dots,\textnormal{ and }h_r(\ub)\!\geqslant\! f_rt$. Notice we have that the scaling of $C_1$ by $t$ is the same as $C_t$ by these definitions.

Let $P_1,\dots,P_t\!\in\!\Q_+^d$ be the vertices of $C_1$. Then, as the coefficients of the hyperplanes are rational, the vertices are rational as well. Let $g$ be the least common multiple of all the denominators of all the components of the $P_j$ in reduced form. Scaling the convex hull by $g$ we have that $gC_1\!=\!C_g$ has as lattice point vertices $gP_1,\dots,gP_t$ by the construction of $g$. Notice that $C_g$ is the minimal such (non-zero) scaling of $C_1$ with lattice point vertices.

Using these lattice point vertices to construct an ideal, let $J$ be the monomial ideal generated by the set $\{\ux^{gP_1},\dots,\ux^{gP_t}\}$ which is well defined since $gP_j\!\in\!\N^d$ for $1\!\leqslant\! j\!\leqslant\! t$. 
We can use the generators of $J$ to construct the Newton polyhedron, i.e. $\textnormal{NP}(J)\!=\!\textnormal{conv}(G(\{\ux^{gP_1},\dots,\ux^{gP_t}\})+\Q_+^d$ where addition is Minkowski addition. Thus, by construction, the Newton Polyhedron of $J$ is $gC_1\!=\!C_g$. Hence we have that $\overline{J}\!=\!I_g$.

To prove the first claim, we need to find the Rees valuations of $J$. Following the method of \cite[Corollary 3.3]{HubSwa08}, the Rees valuations of $J$ are given by the reduced integral form of the bounding faces of the Newton polyhedron. By construction $\textnormal{NP}(J)\!=\!C_g$ and so these bounding faces come from the non-redundant hyperplanes $h_1\!=\!f_1g,\dots,h_r\!=\!f_rg$. To get the reduced integral form of these hyperplanes, let $m_i\!=\!\textnormal{gcd}(a_1^i,\dots,a_d^i,f_ig)$  for $1\!\leqslant\! i \!\leqslant\! r$ be the reducing factor. Notice that since $\textnormal{gcd}(a_1^i,\dots,a_d^i,f_i)\!=\!1$ for $1\!\leqslant\! i \!\leqslant\! r$, then we have that $m_i|g$ for $1\!\leqslant\! i \!\leqslant\! r$. Then the bounding faces of the Newton polyhedron of $J$ in the reduced integral forms are $\frac{1}{m_i}h_i\!=\!\frac{g}{m_i}f_i$ for $1\!\leqslant\! i\!\leqslant\! r$. Thus, the (normalized) Rees valuations of $J$ are $w_i\!=\!\frac{1}{m_i}h_i$ and $w_i(J)\!=\!\frac{g}{m_i}f_i$ for $1\!\leqslant\! i\!\leqslant\! r$.

Then we have that for any $\sigma\!\in\!\Q_+$, $\ux^{\ua}\!\in\! I_{\sigma}$ if and only if $h_i(\ua)\!\geqslant\! \sigma f_i$ for all $1\!\leqslant\! i\!\leqslant\! r$. We can rewrite $h_i(\ua)\!\geqslant\! \sigma f_i$ with the valuations as $w_i(\ux^{\ua})\!\geqslant\! \sigma \frac{f_i}{m_i}$ for each $1\!\leqslant\! i\!\leqslant\! r$. Then notice that we can rewrite this further as 
\[
w_i(\ux^{\ua})\!\geqslant\! \sigma\frac{f_ig}{m_ig}\!=\!\frac{\sigma}{g}w_i(J)\textnormal{ for all }1\!\leqslant\! i\!\leqslant\! r.
\]
Hence, $\ux^{\ua}\!\in\! I_{\sigma}$ if and only if $w_i(\ux^{\ua})\!\geqslant\!\frac{\sigma}{g}w_i(J)\textnormal{ for all }1\!\leqslant\! i\!\leqslant\! r$. Thus, by the rational valuative criterion of \cref{basics}, $I_{\sigma}\!=\!J^{\frac{\sigma}{g}}$ for any $\sigma\!\in\!\Q_+$, finishing (1).

For (2), notice that $\ux^{\ua}\!\in\! I_{\sigma}$ if and only if $\frac{1}{f_i}h_i(\ua)\!\geqslant\! \sigma$ for all $1\!\leqslant\! i\!\leqslant\! r$. By the choice of $e$ we have $\frac{f}{f_i}h_i(\ua)\!\in\!\N$ so that $\frac{1}{f_i}h_i(\ua)\!\in\!\frac{1}{f}\N$ for all $1\!\leqslant\! i\!\leqslant\! r$. Thus, $\frac{1}{f_i}h_i(\ua)\!\geqslant\! \sigma$ implies $\frac{1}{f_i}h_i(\ua)$ is greater than or equal to the nearest element of $\frac{1}{f}\N$, i.e. $\frac{1}{f}\ceil*{f\sigma}$. Hence, $\frac{1}{f_i}h_i(\ua)\!\geqslant\! \sigma$ for each $1\!\leqslant\! i\!\leqslant\! r$ if and only if $\frac{1}{f_i}h_i(\ua)\!\geqslant\! \frac{1}{f}\ceil*{f\sigma}$ for each $1\!\leqslant\! i\!\leqslant\! r$. Thus 
\begin{equation}
\label{Iround}
    I_{\sigma}\!=\!I_{\frac{\ceil*{f\sigma}}{f}}.
\end{equation}
Hence, we can index $\{I_{\sigma}\}_{\sigma\in\Q_+}$ by $\N$ via $\{I_{\frac{n}{f}}\}_{n\in\N}$. Equivalently, we can describe the filtration via the lattice points inside the convex sets $\frac{n}{f}C_1$ for $n\!\in\!\N$.

For (3), we form the canonical indexing of the rational powers of $J$ by looking at the least common multiple of the Rees valuations of $J$ evaluated at $J$, i.e.
\[
e\!=\!\textnormal{lcm}(w_1(J),\dots,w_r(J))\!=\!\textnormal{lcm}(\frac{g}{m_1}f_1,\dots,\frac{g}{m_r}f_r).
\]
Notice that $f_i\frac{g}{m_i}|fg$ for each $1\!\leqslant\! i\!\leqslant\! r$ so that $e|fg$. From the rational valuative criterion of \cref{basics} we can describe the rational powers of $J$ via the lattice points inside the convex sets $\frac{n}{e}\textnormal{NP}(J)$ for $n\!\in\!\N$. Hence, for any $n\!\in\!\N$, $\frac{n}{e}\textnormal{NP}(J)\!=\!\frac{n}{e}C_g\!=\!\frac{ng}{e}C_1$ and using \cref{Iround} we have:
\begin{equation*}
J^{\frac{n}{e}}\!=\!I_{\frac{ng}{e}}\!=\!I_{\frac{1}{f}\cdot\ceil*{\frac{fgn}{e}}}\!=\!I_{\frac{fg}{e}\cdot\frac{n}{f}},
\end{equation*}
where the last equality follows from $e$ dividing $fg$ and thus completing the proof.
\end{proof}

\begin{remark}
\label{extension}
If $\textnormal{gcd}(a_1^i,\dots,a_d^i)\!=\!1$ (or $m_i\!=\!1$) for all $1\!\leqslant\! i\!\leqslant\! r$, then we can improve our understanding of the relationship between $\{J^{\sigma}\}_{\sigma\in\Q_+}$ and $\{I_{\sigma}\}_{\sigma\in\Q_+}$. Specifically, if $m_i\!=\!1$ we have
\[
e\!=\!\textnormal{lcm}(w_1(J),\dots,w_r(J))\!=\!\textnormal{lcm}(f_1g,\dots,f_rg)\!=\!\textnormal{lcm}(f_1,\dots,f_r)g\!=\!fg.
\]
Then, using part (1) and \cref{Iround}, for any $\sigma\!\in\!\Q_+$ we have
\[
J^{\sigma}\!=\!J^{\frac{\ceil*{fg\sigma}}{fg}}\!=\!J^{\frac{1}{g}\cdot\frac{\ceil*{fg\sigma}}{f}}\!=\!I_{\frac{\ceil*{fg\sigma}}{f}}\!=\!I_{g\sigma}
\]
\end{remark}

An important corollary of this theorem establishes a strong connection between integral closure theory and rational powers to symbolic powers.

\begin{corollary}
\label{sqfreeInt}
Let $I$ be a squarefree monomial ideal. Then there exists $g\!\in\!\N$ so that the rational powers of $I^{(g)}$ are the symbolic powers of $I$, that is $I^{(n)}\!=\!(I^{(g)})^{\frac{n}{g}}$ for every $n\!\in\!\N$.
\end{corollary}
\begin{proof}
Since $I$ is squarefree, the minimal primes, $\textnormal{Min}(I)$, are complete intersections and hence we can use their powers to compute the symbolic powers. That is, we can write $I^{(n)}\!=\!\bigcap_{\textnormal{Min}(I)}\p^n$. Then we have for a monomial $\ux^{\ua}\!\in\! R$, $\ux^{\ua}\!\in\! I^{(n)}$ if and only if $\ux^{\ua}\!\in\! \p^n$ for all $\p\!\in\!\textnormal{Min}(I)$. As the $\p$ are monomial primes, we have that $\ux^{\ua}\!\in\! \p^n$ if and only if the degree of $\ux^{\ua}$ supported in the variables of $\p$ is at least $n$. That is, if we write $\p\!=\!(x_{i_1},\dots,x_{i_l})$  for $1\!\leqslant\! i_1\!<\!\dots\!<\!i_l\!\leqslant\! d$, then $\ux^{\ua}\!\in\! \p^n$ if and only if $\alpha_{i_1}+\dots+\alpha_{i_l}\!\geqslant\! n$. Hence, we can associate a hyperplane equation to each $\p$ via $h_{\p}(\ua)\!=\!\alpha_{i_1}+\dots+\alpha_{i_l}$. Thus we have that $\ux^{\ua}\!\in\! I^{(n)}$ if and only if $\ux^{\ua}\!\in\! \p^n$ for all $\p\!\in\!\textnormal{Min}(I)$ if and only if $h_{\p}(\ua)\!\geqslant\! n$ for all $\p\!\in\!\textnormal{Min}(I)$. Hence, $\{I^{(n)}\}_{n\in\N}$ satisfies the hypotheses of \cref{bigThm}, finishing the proof.
\end{proof}

The polyhedron formed by the minimal primes here can be made more general for non-squarefree monomial ideals and is called the \textit{symbolic polyhedron}. For a construction and exposition of the symbolic polyhedron, we refer the reader to \cite{CoopSyP}.

\begin{example}
In $R\!=\!\mathbb{K}[x,y,z]$ let $I\!=\!(xy,yz,zx)\!=\!(x,y)\cap(x,z)\cap(y,z)$. Following the argument of \cref{bigThm}, we have that the symbolic polyhedron is given by the valuations $v_1(x^ay^bz^c)\!=\!a+b$, $v_2(x^ay^bz^c)\!=\!a+c$, and $v_3(x^ay^bz^c)\!=\!b+c$. Then the vertices are $(\frac{1}{2},\frac{1}{2},\frac{1}{2})$, $(1,0,1)$, $(1,1,0)$, and $(0,1,1)$. Hence $g\!=\!2$ so that the rational powers of $I^{(2)}\!=\!(x^2y^2,xyz,x^2z^2,y^2z^2)$ coincide with the symbolic powers of $I$. Specifically, for all $n\!\in\!\N$ we have $I^{(n)}\!=\!(I^{(2)})^{\frac{n}{2}}$.
\end{example}

Using \cref{bigThm} as a characterization of rational powers, we can use the methods of symbolic powers of squarefree monomial ideals to investigate invariants of rational powers. One of the recent methods in this direction comes from \cite{Montano18}.

\section{Convergence of Depth and Regularity via Splittings}

While this section is used to prove facts about rational powers, we present this generalization of \cite{MNB20,Montano18} as it can hold for more general filtrations. Throughout this section, let $R\!=\!\mathbb{K}[x_1,\dots,x_d]$ for some field $\mathbb{K}$, and $I\subset R$ be a monomial ideal.

Using the set-up from \cite{Montano18}, let $m\!\in\!\N$ and denote $R^{\frac{1}{m}}\!=\!\mathbb{K}[x_1^{\frac{1}{m}},\dots,x_d^{\frac{1}{m}}]$, which is a $\frac{1}{m}\N$-graded ring. Notice that  $R^{\frac{1}{m}}$ is isomorphic to $R$ as rings. Let $i:R\to R^{\frac{1}{m}}$ be the natural inclusion given by $i(\ux^{\ua})\!=\!(\ux^{m\ua})^{1/m}$. Following the method of \cite{Montano18} we define the splitting $R$-homomorphism induced by the map $\phi_m^R:R^{\frac{1}{m}}\to R$ given by
\[
\phi_m^R((\ux^{\ua})^{1/m})\!=\!
\begin{cases}
\ux^{\ua/m} & \ua\equiv\underline{0}\,(\textnormal{mod }\,m), \\
0 & \textnormal{otherwise},
\end{cases}
\]
with $\underline{0}\!=\!(0,\dots,0)\!\in\!\N^d$. Notice that $\phi_m^R$ restricted to $R$ is the identity, hence with $i$ forms a split map.

\begin{definition}
\label{asympStab}
We call a filtration
of monomial ideals $\{I_n\}_{n\in\N}$ \textit{asymptotically stable} if it satisfies the following conditions:
\begin{enumerate}
    \item The associated Rees algebra $\bigoplus_{n\in\N}I_nt^n\subset R[t]$ is Noetherian.
    \item For an unbounded sequence $\{m_f\}_{f\in\N}\!\subseteq\!\N$, $i$ and $\phi_{m_f}^R$ above induce a split injection $\iota\!:\!R/I_{n+1}\!\to\! R^{\frac{1}{m_f}}/(I_{nm_f+j})^{\frac{1}{m_f}}$ for all $n,f\!\in\!\N$ and $1\!\leqslant\! j\!\leqslant\! m_f$.
\end{enumerate}
\end{definition}

We can make the second requirement easier to check via the following remark.

\begin{remark}
\label{simp}
For the second condition, to check that $\iota$ is a split injection, since $\phi_m^R$ is a splitting we need only check that $i(I_{n+1})\subseteq (I_{nm+j})^{\frac{1}{m}}$ and that $\phi_m^R((I_{nm+j})^{\frac{1}{m}})\!=\!I_{n+1}$ for infinitely many $m\!\in\!\N$, all $n\!\in\!\N$, and $1\!\leqslant\! j\!\leqslant\! m$. Notice for the former that $i(I_n)\subseteq (I_n^m)^{\frac{1}{m}}\subseteq (I_{nm})^{\frac{1}{m}}$ by the definition of $i$ and the second condition. Hence $i(I_{n+1})\subseteq (I_{nm+m})^{\frac{1}{m}}\subseteq (I_{nm+j})^{\frac{1}{m}}$, and thus it suffices to show the latter condition that $\phi_m^R((I_{nm+j})^{\frac{1}{m}})\!=\!I_{n+1}$. The containment $i(I_{n+1})\subseteq (I_{nm+j})^{\frac{1}{m}}$ also yields that $I_{n+1}\subseteq\phi_m^R((I_{nm+j})^{\frac{1}{m}})$ since $\phi_m^R$ is split and $R$-linear. Hence we need only show the other containment, that $\phi_m^R((I_{nm+j})^{\frac{1}{m}})\subseteq I_{n+1}$.
\end{remark}

\begin{example}
From \cite{MNB20}, we note that the symbolic powers of squarefree monomial ideals form an asymptotically stable filtration.
\end{example}

The following also shows that the rational powers of a monomial ideal form an asymptotically stable family, extending the ideas of \cite{Montano18} and the connection of \cref{sqfreeInt}.

\begin{proposition}
\label{regstable}
Let $I$ be a monomial ideal and $e\!=\!\textnormal{lcm}\{v(I)\,|\,v\!\in\!\mathcal{RV}(I)\}$. Then the rational powers of $I$, $\{I^{\frac{n}{e}}\}_{n\in\N}$ are asymptotically stable.
\end{proposition}
\begin{proof}
From \cref{basics}, the rational powers form a filtration.
Hence we must now show that the associated Rees algebra is Noetherian. From \cite[Lemma 1.1(a)]{HonUlr14}, using our polynomial ring $R$, the algebra $\bigoplus_{n\in\Z}I^{\frac{n}{e}}u^n\subseteq R[u,u^{-1}]$ is a finitely generated $R$-algebra. Hence, its positively graded piece must also be a finitely generated $R$-algebra.

Now, to show condition (2), we use \cref{simp}. Let $\mathcal{RV}(I)\!=\!\{v_1,\dots,v_r\}$, then using \cref{MVPCrit}, let $w_i\!=\!\frac{e}{v_i(I)}v_i$ for each $1\!\leqslant\! i\!\leqslant\! r$. Then let $w_i\!\equiv\!a_1^iX_1+\dots+a_d^iX_d$ be the hyperplane equation corresponding to $w_i$ with $a_k^i\!\in\!\N$ for $1\!\leqslant\! k\!\leqslant\! d$ and $1\!\leqslant\! i\!\leqslant\! r$ as guaranteed in \cref{MVPCrit} and \cref{hyperP}. Let $m,n\!\in\!\N$ and $1\!\leqslant\! j\!\leqslant\! m$. 

Let $(\ux^{\ua})^{1/m}\!\in\! (I^{\frac{nm+j}{e}})^{\frac{1}{m}}$
such that $\phi_m^R((\ux^{\ua})^{1/m})\neq 0$. Then $\ux^{\ua}\!\in\!I^{\frac{nm+j}{e}}$ so that $w_i(\ux^{\ua})\!\geqslant\! nm+j$ for each $1\!\leqslant\! i\!\leqslant\! r$. Let $\ub\!=\!\ua/m\!\in\!\N^d$ as $\ua\equiv\underline{0}\,(\textnormal{mod }\,m)$. In other words $\phi_m^R((\ux^{\ua})^{1/m})\!=\!\ux^{\ub}$. 
So, since $\ux^{\ub}=\ux^{\ua/m}$, applying $\phi_m^R$ yields that for each $1\!\leqslant\! i\!\leqslant\! r$ 
\[
w_i(\ux^{\ua/m})\!=\!a_1^i\frac{\alpha_1}{m}+\dots+a_d^i\frac{\alpha_d}{m}\!=\!\frac{1}{m}(a_1^i\alpha_1+\dots+a_d^i\alpha_d)\!=\!\frac{w_i(\ux^{\ua})}{m}\!\geqslant\! n+\frac{j}{m}.
\]
But since $w_i$ has coefficients in $\N$ we have $w_i(\ub)\!\geqslant\! n+1$ for each $1\!\leqslant\! i\!\leqslant\! r$. Hence $\ux^{\ub}\!\in\! I^{\frac{n+1}{e}}$, finishing the proof.
\end{proof}

The following lemma and theorem are needed to obtain our desired convergences.

\begin{lemma}
\label{ineqLem}
If $\{I_n\}_{n\in\N}$ is an \textit{asymptotically stable} family, then for all $n\!\in\!\N$ and an infinite sequence $\{m_f\}_{f\in\N}\!\subseteq\!\N$ as in \cref{asympStab}:
\begin{enumerate}
    \item[$(1)$] $\textnormal{depth}(R/I_n)\!\leqslant\!\textnormal{depth}(R/I_{\ceil*{\frac{n}{m_f}}})$,
    \item[$(2)$] $a_i(R/I_n)\!\geqslant\! m_fa_i(R/I_{\ceil*{\frac{n}{m_f}}})$ for $0\!\leqslant\! i\!\leqslant\! \textnormal{dim }R/\sqrt{I_1}$.
\end{enumerate}
\end{lemma}
\begin{proof}
Notice that without loss of generality we may assume that the sequence $\{m_f\}_{f\in\N}\!\subseteq\!\N$ is increasing. Now, we have the splitting map $\iota:R/I_{n+1}\to R^{\frac{1}{m_f}}/I_{nm_f+j}^{\frac{1}{m_f}}$ for all $n\!\in\!\N$ and $1\!\leqslant\! j\!\leqslant\! m_f$. Hence, for $0\!\leqslant\! i\!\leqslant\! \textnormal{dim }R/\sqrt{I_1}$, the module $\textnormal{H}_{\mathfrak{m}}^i(R/I_{n+1})$ is a direct summand of $\textnormal{H}_{\mathfrak{m}}^i(R^{\frac{1}{m_f}}/I_{nm_f+j}^{\frac{1}{m_f}})$. Now, by the isomorphism of rings $R\cong R^{\frac{1}{m_f}}$ we have an equivalence of categories (of $R$-modules and $R^{\frac{1}{m_f}}$-modules) and that if we set $\mathfrak{m}\!=\!(x_1,\dots,x_d)$ to be the homogeneous maximal ideal then $\mathfrak{m}R^{\frac{1}{m_f}}\!=\!\mathfrak{m}^{\frac{1}{m_f}}$. Thus we have the following equality:
\begin{equation}
\label{isoLC}
    (\textnormal{H}_{\mathfrak{m}}^i(R/I_{nm_f+j}))^{\frac{1}{m_f}}\!=\!\textnormal{H}_{\mathfrak{m}^{\frac{1}{m_f}}}^i(R^{\frac{1}{m_f}}/I_{nm_f+j}^{\frac{1}{m_f}})\!=\!\textnormal{H}_{\mathfrak{m}}^i(R^{\frac{1}{m_f}}/I_{nm_f+j}^{\frac{1}{m_f}}).
\end{equation}
Hence $\textnormal{H}_{\mathfrak{m}}^i(R/I_{nm_f+j})\!=\!0$ implies $\textnormal{H}_{\mathfrak{m}}^i(R^{\frac{1}{m_f}}/I_{nm_f+j}^{\frac{1}{m_f}})\!=\!0$ and hence, as it is a direct summand, $\textnormal{H}_{\mathfrak{m}}^i(R/I_{n+1})\!=\!0$. Thus
\[
\textnormal{depth}(R/I_{n+1})\!\leqslant\!\textnormal{depth}(R/I_{nm_f+j})
\]
proving the first part.

Again from the splitting of $\iota$ and hence the direct summand of local cohomology, and \cref{isoLC} above, we have
\[
a_i(R/I_{n+1})\!\leqslant\! a_i(R^{\frac{1}{m_f}}/I_{nm_f+j}^{\frac{1}{m_f}})\!=\!\frac{1}{m_f}a_i(R/I_{nm_f+j})
\]
showing the second part.
\end{proof}

\begin{theorem}
\label{depthReg}
If $\{I_n\}_{n\in\N}$ is an \textit{asymptotically stable} family, then
\begin{enumerate}
    \item[$(1)$] $\lim\limits_{n \to \infty}\frac{\textnormal{reg}(I_n)}{n}$ exists,
    \item[$(2)$] $\lim\limits_{n \to \infty}\textnormal{depth}(R/I_n)$ exists.
\end{enumerate}
\end{theorem}
\begin{proof}
Let $\{m_f\}_{e\in\N}$ be the sequence of natural numbers that satisfy the infinitely many splittings of \cref{asympStab}. Similarly to before, we may assume without loss of generality that the sequence is increasing.

For (1), note that
\begin{equation}
\label{regLim}
\begin{split}
    \lim\limits_{n \to \infty}\frac{\textnormal{reg}(I_n)}{n}
    & \!=\!\lim\limits_{n \to \infty}\frac{\textnormal{reg}(R/I_n)+1}{n} \\
    & \!=\!\lim\limits_{n \to \infty}\frac{\max\{a_i(R/I_n)+i+1\,|\,0\!\leqslant\! i\!\leqslant\! \textnormal{dim }R/\sqrt{I_1}\}}{n} \\
    & \!=\!\lim\limits_{n \to \infty}\frac{\max\{a_i(R/I_n)\,|\,0\!\leqslant\! i\!\leqslant\! \textnormal{dim }R/\sqrt{I_1}\}}{n} \\
    & \!=\!\max\{\lim\limits_{n \to \infty}\frac{a_i(R/I_n)}{n}\,|\,0\!\leqslant\! i\!\leqslant\! \textnormal{dim }R/\sqrt{I_1}\}. \\
\end{split} 
\end{equation}
Thus we set $\alpha_n\!=\!\max_i\{a_i(R/I_n)\}$. Then, as the associated Rees algebra is Noetherian, using the lemma inside the proof from \cite[Proof of Theorem 4.3]{CHT99}, we know that $\textnormal{reg}(I_n)$ must be quasi-linear for large enough $n$. Hence, there are $c_1,\dots,c_r,b_1,\dots,b_r\!\in\!\N$ such that $\textnormal{reg}(R/I_n)\!=\!c_jn+b_j$ for $n\equiv j \,(\textnormal{mod }\,r)$ for $n\!\gg\!0$. By \cref{regLim} and the quasi-linearity we have $\lim\limits_{k\to\infty}\frac{\alpha_{rk+j}}{rk+j}\!=\!c_j$ for each $j$. From this, fix $1\!\leqslant\! i,j\!\leqslant\! r$ and $m_f\!>\!r$. 
Let $\epsilon\!>\!0$ and $t\!\in\!\N$ such that $|c_j-\frac{\alpha_{rk+j}}{rk+j}|\!<\!\epsilon$ for all $k\!\geqslant\! t$. Thus, $c_j-\frac{\alpha_{rk+j}}{rk+j}\!<\!\epsilon$. From \cref{ineqLem} part (2) we have that
\[
\alpha_{rk+j}\!\leqslant\!\frac{\alpha_{m_{fs}(rk+j-1)+b}}{m_{fs}}
\]
for every $s\!\in\!\N$ and $1\!\leqslant\! b\!\leqslant\! m_{fs}$. Thus,
\begin{equation}
\label{ineqReg}
    c_j-\epsilon\!\leqslant\!\frac{\alpha_{rk+j}}{rk+j}\!\leqslant\!\frac{\alpha_{m_{fs}(rk+j-1)+b}}{m_{fs}(rk+j)}\!\leqslant\!\frac{\alpha_{m_{fs}(rk+j-1)+b}}{m_{fs}(rk+j-1)+b}.
\end{equation}
Because \cref{ineqReg} holds for all $1\!\leqslant\! b\!\leqslant\! m_{fs}$ and that $r\!<\!m_f\!\leqslant\! m_{fs}$, we can find infinitely many $s,k$, and $b$ such that $m_{fs}(rk+j-1)+b\equiv i \,(\textnormal{mod }\,r)$. We use the pairs in \cref{ineqReg} to yield in the limit that $c_j-\epsilon\!\leqslant\! c_i$ for all $\epsilon$. Hence $c_j\!\leqslant\! c_i$ and as $i$ and $j$ were arbitrary, we have $c_1\!=\!\dots\!=\!c_r$, showing that the regularity limit exists.

For (2), again as the associated Rees algebra is Noetherian, there is a $k\!\in\!\N$ such that $I_{n+k}\!=\!I_nI_k$ for every $n\!\geqslant\! k$ (see \cite[Remark 2.4.3]{Rat79}). Thus, for $0\!\leqslant\! j\!\leqslant\! k-1$ there are $d_j,h_j\!\in\!\N$ with
\[
\textnormal{depth}(I_{(n+1)k+j})\!=\!\textnormal{depth}((I_{k})^nI_{k+j})\!=\!d_j
\]
so that $\textnormal{depth}(R/I_{(n+1)k+j})\!=\!d_j-1$ for $n\!\geqslant\! h_j$ from \cite[Theorem 1.1]{HH05}.

Let $d\!=\!\min_n\{\textnormal{depth}(R/I_n)\}$ and fix $l\!\in\!\N$ with $d\!=\!\textnormal{depth}(R/I_l)$. Let $f\!\in\!\N$ such that $m_f\!>\!k$ and $m_f(l-1)\!>\!(h_j+1)k$ for $0\!\leqslant\! j\!\leqslant\! k-1$. Now, \cref{ineqLem} part (2) implies that
\[
\textnormal{depth}(R/I_{m_f(l-1)+i})\!\leqslant\!\textnormal{depth}(R/I_{l})
\]
for $1\!\leqslant\! i\!\leqslant\! m_f$. By choice of $m_f$, for each $0\!\leqslant\! j\!\leqslant\! k-1$ there are $t\!\geqslant\! h_j+1$ and $1\!\leqslant\! i\!\leqslant\! m_f$ with $m_f(l-1)+i\!=\!tk+j$. Then
\[
d\!=\!\textnormal{depth}(R/I_l)\!\geqslant\!\textnormal{depth}(R/I_{m_f(l-1)+i})\!=\!\textnormal{depth}(R/I_{tk+j})\!=\!d_j-1.
\]
Hence, $d\!=\!d_j-1$ for every $0\!\leqslant\! j\!\leqslant\! k-1$, showing the convergence.
\end{proof}

\section{Asymptotics of Rational Powers}

This section focuses on using the connection between rational powers and symbolic powers to find the convergence of Stanley depth, and the connection with integral closure powers to find the convergence of length for some local cohomology modules. First, as we can gleam from the previous section, we present the proof of \cref{regdepthRP}.

\begin{theorem}
\label{verLimits}
If $I$ is any monomial ideal and $e\!=\!\textnormal{lcm}\{v(I)\,|\,v\!\in\!\mathcal{RV}(I)\}$, then 
\begin{enumerate}
    \item[$(1)$] $\lim\limits_{n \to \infty}\frac{\textnormal{reg}(I^{\frac{n}{e}})}{n}$ exists and is equal to $\frac{1}{e}\lim\limits_{n \to \infty}\frac{\textnormal{reg}(\overline{I^n})}{n}$.
    \item[$(2)$] $\lim\limits_{n \to \infty}\textnormal{depth}(R/I^{\frac{n}{e}})$ exists and is equal to $d-\ell(I)$, where $\ell(I)$ is the analytic spread of $I$.
\end{enumerate}
\end{theorem}
\begin{proof}
From \cref{regstable} and \cref{depthReg}, both of these limits exist. 
For the regularity, notice that for $n\!\in\!\N$ that 
\[
\frac{\textnormal{reg}(\overline{I^n})}{n}\!=\!\frac{\textnormal{reg}(I^{\frac{ne}{e}})e}{ne}\!=\!e\frac{\textnormal{reg}(I^{\frac{ne}{e}})}{ne}
\]
so that if $\lim\limits_{n \to \infty}\frac{\textnormal{reg}(I^{\frac{n}{e}})}{n}\!=\!p_e$ and $\lim\limits_{n \to \infty}\frac{\textnormal{reg}(\overline{I^n})}{n}\!=\!p$ then $ep_e\!=\!p$. From this equality, $p_e\!=\!\frac{1}{e}p_1$, finishing (1).

For depth, notice that from the correspondence $\overline{I^n}\!=\!I^{\frac{ne}{e}}$, $\{\textnormal{depth}(R/\overline{I^n})\}_{n\in\N}$ appears as subsequences of the $I^{\frac{n}{e}}$ depth sequence. Hence both sequences have the same limit and so by \cite[Lemma 1.5]{Trung18}, we have finished (2).
\end{proof}

Notice due to the connection with symbolic powers of squarefree monomial ideals, we can conclude a convex-geometric computation of the symbolic analytic spread.

\begin{remark}
\label{symSpread}
Let $I$ be a squarefree monomial ideal. Then there exists a $g\!\in\!\N$ such that the symbolic analytic spread $\ell_s(I)\!=\!\ell(I^{(g)})$.
Notice that this allows us to have $\ell_s(I)$ computed by using the Symbolic Polyhedron.
To see this, notice that from \cref{sqfreeInt} we have that such a $g$ exists. Then the rational powers of $I^{(g)}$ are the symbolic powers of $I$. By \cref{verLimits} we have that
\[
\lim\limits_{n \to \infty}\textnormal{depth}(R/I^{(n)})\!=\!\lim\limits_{n \to
\infty}\textnormal{depth}(R/(I^{(g)})^{\frac{n}{g}})\!=\!d-\ell(I^{(g)}).
\]
On the other hand, from \cite[Theorem 3.6]{Fak20}
\[
\lim\limits_{n \to \infty}\textnormal{depth}(R/I^{(n)})\!=\!d-\ell_s(I).
\]
and thus $\ell_s(I)\!=\!\ell(I^{(g)})$. 

Then, from \cite[Theorem 2.3]{B-A03} we can compute $\ell(I^{(g)})$ by the maximal dimension of a compact face of $\textnormal{NP}(I^{(g)})$ plus one. However, as we have seen, $\textnormal{NP}(I^{(g)})$ is a multiple of the symbolic polyhedron, $\textnormal{SP}(I)$, and thus, since taking multiples of a convex hull does not change the dimension of the faces, we can look at the dimension of the faces of $\textnormal{SP}(I)$ to compute $\ell(I^{(g)})\!=\!\ell_s(I)$.
\end{remark}

Notice also that \cref{verLimits} allows us to use the connection between the rational and integral powers. We can now use this connection for Stanley Depth. Most of the study of Stanley depth of integral closure powers has been limited (e.g. \cite{Fak19}), but this connection allows us to conclude new facts about integral closure powers of monomial ideals. We now prove \cref{sdepthRP}.

\begin{theorem}
\label{sdepth}
If $I$ is a monomial ideal and $e\!=\!\textnormal{lcm}\{v(I)\,|\,v\!\in\!\mathcal{RV}(I)\}$. Then the limits $\lim\limits_{n \to \infty}\textnormal{sdepth}(R/I^{\frac{n}{e}})$ and $\lim\limits_{n \to \infty}\textnormal{sdepth}(I^{\frac{n}{e}})$ exist. In particular, these limits must exist for $\{\overline{I^n}\}_{n\in\N}$. Furthermore:
\begin{equation*}
\begin{split}
    \lim\limits_{n \to \infty}\textnormal{sdepth}(R/I^{\frac{n}{e}})\!=\!\min_n\textnormal{sdepth}(R/\overline{I^n})\textnormal{, and} \\
    \lim\limits_{n \to \infty}\textnormal{sdepth}(I^{\frac{n}{e}})\!=\!\min_n\textnormal{sdepth}(\overline{I^n}).
\end{split}
\end{equation*}
\end{theorem}
\begin{proof}
Using \cref{MVPCrit}, suppose that the Rees valuations of $I$ are $v_1,\dots,v_r$, let $w_i\!=\!\frac{e}{v_i(I)}v_i$ for each $1\!\leqslant\! i\!\leqslant\! r$. 

For ease of notation, let $f$ denote a monomial of $R$. Using \cite[Lemma 4.1, Theorem 4.2]{Fak20} let $m\!\in\!\N$ and $k\!\leqslant\! m$, then for $j$ with $m-k\!\leqslant\! j\!\leqslant\! m$ we claim that $f\!\in\! I^{\frac{m}{e}}$ if and only if $f^{k+1}\!\in\! I^{\frac{km+j}{e}}$. Indeed, from \cref{basics}, $f\!\in\! I^{\frac{m}{e}}$ implies $f^{k+1}\!\in\! (I^{\frac{m}{e}})^{k+1}\subseteq I^{\frac{m(k+1)}{e}}\subseteq I^{\frac{km+j}{e}}$. 

Conversely, assume $f^{k+1}\!\in\! I^{\frac{km+j}{e}}$ and $f\notin I^{\frac{m}{e}}$. Then there is an $1\!\leqslant\! i\!\leqslant\! r$ such that $w_i(f)\!<\!m$. Since, by \cref{hyperP} and \cref{MVPCrit}, we can consider $w_i$ as a hyperplane equation with coefficients in $\N$ we must have $w_i(f)\!\leqslant\! m-1$. Hence
\[
w_i(f^{k+1})\!\leqslant\! (k+1)(m-1)\!=\!km+m-k-1\!<\!km+(m-k)\!\leqslant\! km+j
\]
which implies that $f^{k+1}\notin I^{\frac{km+j}{e}}$ a contradiction; proving the equivalence. 

Thus, by \cite[Theorem 4.2]{Fak20}, we have
\begin{equation}
\label{sdLin}
\textnormal{sdepth}(I^{\frac{m}{e}})\!\geqslant\!\textnormal{sdepth}(I^{\frac{km+j}{e}})\textnormal{ and } \textnormal{sdepth}(R/I^{\frac{m}{e}})\!\geqslant\!\textnormal{sdepth}(R/I^{\frac{km+j}{e}}).
\end{equation}
 
Similarly, for any $s,k\!\in\!\N$ we have $f\!\in\! I^{\frac{s}{e}}$ if and only if $f^s\!\in\! I^{\frac{ks}{e}}$ by the monomial valuation inequalities. Indeed, $w_i(f)\!\geqslant\!\frac{s}{e}$ for each $1\!\leqslant\! i\!\leqslant\! r$ if and only if $w_i(f^k)\!\geqslant\!\frac{sk}{e}$ for each $1\!\leqslant\! i\!\leqslant\! r$ since valuations are a multiplicative group homomorphism. Thus, from \cite[Theorem 3.1]{Fak17} we must have
\begin{equation}
    \label{sdepthM}
    \textnormal{sdepth}(R/I^{\frac{ks}{e}})\!\leqslant\!\textnormal{sdepth}(R/I^{\frac{s}{e}}) \textnormal{ and } \textnormal{sdepth}(I^{\frac{ks}{e}})\!\leqslant\!\textnormal{sdepth}(I^{\frac{s}{e}})
\end{equation}

Set $m\!=\!\min\{\textnormal{sdepth}(R/I^{\frac{k}{e}})\,|\,k\!\in\!\N\}$ and $t\!=\!\min\{s\,|\,\textnormal{sdepth}(R/I^{\frac{s}{e}})\!=\!m\}$. 
If $t\!=\!1$ then from \cref{sdepthM} we have 
\[
\textnormal{sdepth}(R/I^{\frac{1}{e}})\!\leqslant\!\textnormal{sdepth}(R/I^{\frac{k}{e}})\!=\!\textnormal{sdepth}(R/I^{\frac{k\cdot1}{e}})\!\leqslant\!\textnormal{sdepth}(R/I^{\frac{1}{e}})
\]
proving the limit. So, suppose that $t\!>\!1$. Then $\textnormal{sdepth}(R/I^{\frac{t^2-t}{e}})\!\leqslant\!\textnormal{sdepth}(R/I^{\frac{t}{e}})\!=\!m$ again from \cref{sdepthM}. For every $k\!>\!t^2-t$ write $k\!=\!lt+j$ with $1\!\leqslant\! j\!\leqslant\! t$. Then notice that $l\!\geqslant\! t-1$ since $lt+j\!>\!t^2-t$. Then by \cref{sdLin} we have that
\[
\textnormal{sdepth}(R/I^{\frac{k}{e}})\!=\!\textnormal{sdepth}(R/I^{\frac{lt+j}{e}})\!\leqslant\!\textnormal{sdepth}(R/I^{\frac{t}{e}})\!=\!m
\]
and thus $\textnormal{sdepth}(R/I^{\frac{k}{e}})\!=\!m$ for all $k\!>\!t^2-t$, proving the convergence. The same argument then holds for $\textnormal{sdepth}(I^{\frac{n}{e}})$ as well.

For $\{\overline{I^n}\}_{n\in\N}$ and the limit equalities, notice that the correspondence $\overline{I^n}\!=\!I^{\frac{ne}{e}}$ yields $\{\textnormal{sdepth}(R/\overline{I^n})\}_{n\in\N}$ as a subsequence of the sequence $\{\textnormal{sdepth}(R/I^{\frac{n}{e}})\}_{n\in\N}$. Hence both sequences have the same limit, and thus
\[
\lim\limits_{n \to \infty}\textnormal{sdepth}(R/I^{\frac{n}{e}})\!=\!\min_n\textnormal{sdepth}(R/I^{\frac{n}{e}})\!\leqslant\!\min_n\textnormal{sdepth}(R/\overline{I^n})\!\leqslant\!\lim\limits_{n \to \infty}\textnormal{sdepth}(R/\overline{I^n})
\]
completes the proof.
\end{proof}

We can use the fact that rational powers are integrally closed to conclude other desirable properties of a filtration, for example, the finiteness of associated primes. We will present this argument in full generality.

\begin{proposition}
\label{assFilt}
If $\{I_n\}_{n\in\N}$ is a filtration whose associated Rees algebra is Noetherian and such that $I_n$ is integrally closed for all $n\!\gg\! 0$, then there exists an $m\!\in\!\N$ with $\textnormal{Ass}(R/I_n)\subseteq\textnormal{Ass}(R/I_m)$ for all $n\!\geqslant\! m$. Hence $\cup_{n\in\N}\textnormal{Ass}(R/I_n)$ is finite.
\end{proposition}
\begin{proof}
By \cite[2.4.4]{Rat79}, for all $c\!\gg\!0$ and $n\!\geqslant\!0$ we have that $I_{cn}\!=\!I_c^n$. Now let $m\!\gg\!0$ such that the preceding property holds and that $I_k$ is integrally closed for all $k\!\geqslant\! m$. Then for $n_m\!\gg\!0$ we have
\[
\textnormal{Ass}(R/I_{mn})\!=\!\textnormal{Ass}(R/I_m^n)\!=\!\textnormal{Ass}(R/\overline{I_m^n})\!=\!\textnormal{Ass}(R/\overline{I_m^{n_m}})\!=\!\textnormal{Ass}(R/I_{mn_m})
\]
for all $n\!\geqslant\! n_m$ by the stabilizing of associated primes for integral closure powers from \cite[6.8.8]{HSInt}. Now let $r\!\geqslant\! mn_m$, then as $I_r$ is integrally closed we have
\[
\textnormal{Ass}(R/I_r)\subseteq\textnormal{Ass}(R/\overline{I_r^{n_r}})\!=\!\textnormal{Ass}(R/I_{rn_r})\!=\!\textnormal{Ass}(R/I_{rn_rmn_m})\!=\!\textnormal{Ass}(R/I_{mn_m})
\]
finishing the proof.
\end{proof}

Clearly, we can apply \cref{regstable} to rational powers as follows.

\begin{corollary}
\label{assInt}
Let $I$ be a monomial ideal. Then $\cup_{n\in\N}\textnormal{Ass}(R/I^{\frac{n}{e}})$ is finite with $e\!=\!\textnormal{lcm}\{v(I)\,|\,v\!\in\!\mathcal{RV}(I)\}$.
\end{corollary}

We next show the existence of the limit of the lengths of local cohomology modules. In order to do this, we will be using the method of \cite{DaoMon17}. First we must establish some lemmas that will allow us to use their results.

\begin{lemma}
\label{localRat}
For a commutative ring $A$, multiplicatively closed set $W\subseteq A$, ideal $I\subseteq A$, and $\alpha\!\in\!\Q_+$, we have $W^{-1}I^{\alpha}\!=\!(W^{-1}I)^{\alpha}$.
\end{lemma}
\begin{proof}
Write $\alpha\!=\!\frac{p}{q}$ for coprime $p,q\!\in\!\N$. We will use \cite[Proposition 1.1.4]{HSInt} and their argument for persistence.
We will first show that taking rational powers is persistent, that is, if $\phi:A\to S$ is a ring homomorphism to some commutative ring $S$, then $\phi(I^{\alpha})\subseteq(\phi(I)S)^{\alpha}$. Indeed if $s\!\in\!\phi(I^{\alpha})$, then $s\!=\!\phi(r)$ for some $r\!\in\! I^{\alpha}$. Then there is an integral dependence equation of $r^q$ over $I^p$ since $r^q\!\in\!\overline{I^p}$. Then applying $\phi$ to the equation gives an integral dependence equation of $\phi(r^q)\!=\!s^q$ over $(\phi(I)S)^p$ so that $s\!\in\!(\phi(I)S)^{\alpha}$.

Now for the localization case, persistence yields one containment, for the other let $x\!\in\!(W^{-1}I)^{\alpha}$. Since $x^q\!\in\!\overline{(W^{-1}I)^p}\!=\!W^{-1}\overline{I^p}$ there exists a $w\!\in\! W$ with $wx^q\!\in\!\overline{I^p}$ and with $wx\!\in\! A$. Thus, $w^{q-1}wx^q\!=\!(wx)^q\!\in\!\overline{I^p}$ so that $wx\!\in\! I^{\frac{p}{q}}\!=\!I^{\alpha}$ and hence $x\!\in\! W^{-1}I^{\alpha}$, finishing the proof.
\end{proof}

Notice here that we do not need a polynomial ring for the localization, making the lemma a general result on rational powers. We now must appeal to the convex geometry of monomial ideals as used in \cite{DaoMon17}.

\begin{lemma}
\label{technical}
Let $I_1,\dots,I_r,J_1,\dots,J_s$ be monomial ideals in $R$ and $e\!\in\!\N$. For every $n\!\geqslant\! 1$ consider the set
\[
S_n\!=\!\{\ua\!\in\!\N^d\,|\,\ux^{\ua}\!\in\! I_i^{\frac{n}{e}}\textnormal{ for every }1\!\leqslant\! i\!\leqslant\! r\textnormal{, and }\ux^{\ua}\notin J_i^{\frac{n}{e}}\textnormal{ for every }1\!\leqslant\! i\!\leqslant\! s\}.
\]
Assume $|S_n|\!<\!\infty$ for $n\!\gg\!0$. Then $\lim\limits_{n \to \infty}\frac{|S_n|}{n^d}$ exists and is equal to a rational number.
\end{lemma}
\begin{proof}
Let $\Gamma:=\!\cap_{i\!=\!1}^r\frac{1}{e}\textnormal{NP}(I_i)$ and $\Gamma_i:=\!\Gamma\cap\frac{1}{e}\textnormal{NP}(J_i)$ for $1\!\leqslant\! i\!\leqslant\! s$. Furthermore, let $nC:=\!n\Gamma-\cup_{i\!=\!1}^sn\Gamma_i$. Then notice that $\ua\!\in\! S_n$ if and only if $\ua\!\in\!\frac{n}{e}\textnormal{NP}(I_i)$ for each $1\!\leqslant\! i\!\leqslant\! r$ and $\ua\notin\frac{n}{e}\textnormal{NP}(J_i)$ for each $1\!\leqslant\! i\!\leqslant\! s$. Using the new notation, $\ua\!\in\! S_n$ if and only if $\ua\!\in\! nC$. Hence, by assumption $C$ is bounded and so by \cite[Lemma 4.1]{DaoMon17} the limit exists and is rational.
\end{proof}

Before proceeding to apply this lemma, we must establish some notation from \cite{DaoMon17}. Let $[d]\!=\!\{1,\dots,d\}$ and $F\subseteq[d]$. Then we let $\pi_F:R\to R$ be the map defined by $\pi_F(x_i)\!=\!1$ if $i\!\in\! F$ and $\pi_F(x_i)\!=\!x_i$ otherwise. For an ideal $I$ we say $I_F:=\!\pi_F(I)$.

For $\ua\!\in\!\Z^d$ we let $G_{\ua}\!=\!\{i\,|\,\ua_i\!<\!0\}$ and $\ua^+\!=\!(\ua_1^+,\dots,\ua_d^+)$ where $\ua_i^+\!=\!\ua_i$ if $i\notin G_{\ua}$ and $\ua_i^+\!=\!0$ otherwise. For a monomial ideal $I$ we let $\Delta_{\ua}(I)$ be the simplicial complex of all subsets $F$ of $[d]-G_{\ua}$ such that $\ux^{\ua^+}\notin I_{F\cup G_{\ua}}$. Finally, for a monomial ideal $I$, let $\Delta(I)$ be the simplicial complex of all $F\subseteq[d]$ with $\Pi_{i\!\in\! F}x_i\notin\sqrt{I}$. We are now ready to prove \cref{lengthRP}.

\begin{theorem}
\label{lcLim}
Let $I$ be a monomial ideal and $e\!=\!\textnormal{lcm}\{v(I)\,|\,v\!\in\!\mathcal{RV}(I)\}$. Assume that $\lambda(\textnormal{H}_{\mathfrak{m}}^i(R/I^{\frac{n}{e}}))\!<\!\infty$ for $n\!\gg\!0$. Then the limit
\[
\lim\limits_{n \to \infty}\frac{\lambda(\textnormal{H}_{\mathfrak{m}}^i(R/I^{\frac{n}{e}}))}{n^d}
\]
exists and is rational. In particular, this limit must exist for $\{I^{(n)}\}_{n\in\N}$ when $I$ is squarefree. Furthermore $\lim\limits_{n \to \infty}\frac{\lambda(\textnormal{H}_{\mathfrak{m}}^i(R/I^{\frac{n}{e}}))}{n^d}\!=\!\frac{1}{e^d}\lim\limits_{n \to \infty}\frac{\lambda(\textnormal{H}_{\mathfrak{m}}^i(R/\overline{I^n}))}{n^d}$.
\end{theorem}
\begin{proof}
From \cref{regstable} the associated Rees algebra of the rational powers is Noetherian, and so by the application of Takayama's formula from \cite[Theorem 3.8]{DaoMon17} we have that
\[
\lambda(\textnormal{H}_{\mathfrak{m}}^i(R/I^{\frac{n}{e}}))\!=\!\Sigma_{\Delta'\subseteq\Delta(I)}\textnormal{dim}_k\tilde{H}_{i-1}(\Delta',k)f_{\Delta'}(n)
\]
where $f_{\Delta'}(n)\!=\!|\{\ua\!\in\!\N^d\,|\,\Delta_{\ua}(I^{\frac{n}{e}})\!=\!\Delta'\}|$. We claim that for each $\Delta'\subset\Delta(I)$ such that $\tilde{H}_{i-1}(\Delta',k)\neq0$ the limit $\lim\limits_{n\to\infty}\frac{f_{\Delta'}(n)}{n^d}$ exists and is rational, which will finish the proof. By assumption we must have $f_{\Delta'}(n)\!<\!\infty$ for $n\!\gg\!0$ whenever $\tilde{H}_{i-1}(\Delta',k)\neq0$. For $\ua\!\in\!\N^d$ we have $\Delta_{\ua}(I^{\frac{n}{e}})\!=\!\Delta'$ if and only if $\ux^{\ua}\notin (I^{\frac{n}{e}})_F$ for every facet $F$ of $\Delta'$, and $\ux^{\ua}\!\in\!(I^{\frac{n}{e}})_G$ for every minimal non-face $G$ of $\Delta'$. Then by \cref{localRat} we have $(I^{\frac{n}{e}})_F\!=\!(I_F)^{\frac{n}{e}}$ and $(I^{\frac{n}{e}})_G\!=\!(I_G)^{\frac{n}{e}}$ for each $F$ and $G$. Thus, we may use \cref{technical} to show the existence of the limit.

By \cref{sqfreeInt}, the limit must exist under the same assumptions for $\{I^{(n)}\}_{n\in\N}$ when $I$ is squarefree. Notice that the correspondence $\overline{I^n}\!=\!I^{\frac{ne}{e}}$ yields the limit equality via the argument of \cref{verLimits} part (1). 
\end{proof}

\section*{Acknowledgements}
The author would like to thank his advisor, Jonathan Monta{\~n}o, for the many invaluable suggestions, comments, and discussions he provided. The author would also like to thank Louiza Fouli for her helpful comments.

\bibliographystyle{amsplain}

\providecommand{\bysame}{\leavevmode\hbox to3em{\hrulefill}\thinspace}
\providecommand{\MR}{\relax\ifhmode\unskip\space\fi MR }
\providecommand{\MRhref}[2]{%
  \href{http://www.ams.org/mathscinet-getitem?mr=#1}{#2}
}
\providecommand{\href}[2]{#2}

\end{document}